\documentclass[letter,9pt]{article}
\usepackage{times,url}
\usepackage{amsmath,amssymb,amsthm,amsfonts}

\newtheorem{theorem}{Theorem}[section]
\newtheorem{corollary}{Corollary}[section]
\newtheorem{lemma}{Lemma}[section]

\newtheorem{remark}{Remark}[section]

\begin{document}
\setcounter{page}{1} 
\vspace{10mm}

\begin{center}
{\LARGE \bf  Analysis of the Ratio $D(n)/n$}
\vspace{8mm}

{\large \bf Jose Arnaldo B. Dris}
\vspace{3mm}

Department of Mathematics and Physics, Far Eastern University \\ 
Nicanor Reyes Street, Sampaloc, Manila, Philippines \\
e-mails: \url{josearnaldobdris@gmail.com}, \url{jadris@feu.edu.ph}
\vspace{2mm}

\end{center}
\vspace{10mm}

\noindent
{\bf Abstract:} In this note, we investigate properties of the ratio $D(n)/n$, which we will call the deficiency index.  We will discuss some concepts recast in the language of the deficiency index, based on similar considerations in terms of the abundancy index. 
\vspace{10mm}

\section{Introduction}
If $n$ is a positive integer, then we write $\sigma(n)$ for the sum of the divisors of $n$.  A number $n$ is \emph{perfect} if $\sigma(n)=2n$.  We call $M$ \emph{almost perfect} if $\sigma(M)=2M-1$.  We say $k$ is \emph{deficient} if $\sigma(k)<2k$, and we call $m$ \emph{abundant} if $\sigma(m)>2m$.  We denote the \emph{abundancy index} $I$ of the positive integer $w$ as $I(w) = \sigma(w)/w$.  We also denote the deficiency $D$ of the positive integer $x$ as $D(x) = 2x-\sigma(x)$ \cite{OEIS-A033879}.  (In this case, if $D(x)>0$ we say that $x$ is deficient by $D(x)$, since the last equation can be rewritten as $\sigma(x)=2x-D(x)$.  Similarly, if $D(x)<0$ we say that $x$ is abundant by $D(x)$.  Of course, if $D(x)=0$ then $x$ is perfect.)  Lastly, we will call the ratio $D(x)/x$ as the \emph{deficiency index} of $x$, and will denote it by $d(x)=D(x)/x$.  Notice that we have the equation
$$2 - I(x) = 2 - \frac{\sigma(x)}{x} = \frac{2x-\sigma(x)}{x} = \frac{D(x)}{x} = d(x).$$

In his undergraduate honors thesis \cite{Ludwick}, Ludwick analyzed the properties of the ratio $I(n)=\sigma(n)/n$.

\section{On a Criterion for Deficient Numbers in Terms of the Abundancy and Deficiency Indices}
In the preprint \cite{Dris1}, Dris proves that $n$ is deficient by $D(n)>1$ if and only if the following bounds hold:
\begin{theorem}\label{Theorem1}
$\sigma(n)=2n-D(n)$ and $D(n)>1$ if and only if
$$\frac{2n}{n+D(n)}<I(n)<\frac{2n+D(n)}{n+D(n)}.$$
\end{theorem}

We will prove the following version of Theorem \ref{Theorem1} here:
\begin{theorem}\label{Theorem2}
$\sigma(M)=2M-1$ if and only if
$$\frac{2M}{M+1} \leq I(M) < \frac{2M+1}{M+1}.$$
\end{theorem}

\begin{proof}
Rewriting the bounds, we obtain
$$\frac{2M}{M+1}=\frac{2(M+1)}{M+1}-\frac{2}{M+1}=2-\frac{2}{M+1}$$
and
$$\frac{2M+1}{M+1}=\frac{2(M+1)}{M+1}-\frac{1}{M+1}=2-\frac{1}{M+1}.$$
Now, $\sigma(M)=2M-1$ if and only if $I(M)=\sigma(M)/M=2-(1/M)$.  We want to show that
$$2-\frac{2}{M+1} \leq I(M)=2-\frac{1}{M} < 2-\frac{1}{M+1}.$$
Cancelling $2$ and rearranging, we get
$$\frac{1}{M+1} < \frac{1}{M} \leq \frac{2}{M+1},$$
which is trivially true as
$$M < M+1 \leq 2M$$
holds, where the inequality on the right follows from $M \geq 1$.  This proves one direction of the theorem.
Now, suppose that
$$2-\frac{2}{M+1} \leq I(M) < 2-\frac{1}{M+1}.$$
This implies that
$$\frac{1}{M+1} < 2 - I(M) \leq \frac{2}{M+1}$$
from which we obtain
$$0 < \frac{M}{M+1} < D(M) \leq \frac{2M}{M+1}.$$
We claim that $D(M)=1$.  Suppose to the contrary that $D(M) \geq 2$.  Then we have
$$2 \leq D(M) \leq \frac{2M}{M+1}$$
resulting in the contradiction $2(M+1)=2M+2 \leq 2M$.  Hence, $D(M)=1$, and we are done.
\end{proof}

In particular, the criterion in Theorem \ref{Theorem1} can be rewritten in terms of the deficiency index, as follows:
$\sigma(n)=2n-D(n)$ and $D(n)>1$ if and only if
$$\frac{2}{1+d(n)}<I(n)<\frac{2+d(n)}{1+d(n)}.$$

As an application of the criterion in Theorem \ref{Theorem1}, we can prove that primes, powers of primes, and products of two distinct odd prime powers are deficient.

First, we dispose of two technical lemmas.

\begin{lemma}\label{FirstLemma}
If $x \mid y$, then $d(y) \leq d(x)$.
\end{lemma}

\begin{proof}
Suppose that $x \mid y$.  This implies that $I(x) \leq I(y)$, from which it follows that
$$d(y)=\frac{D(y)}{y}=2-I(y) \leq 2-I(x)=\dfrac{D(x)}{x}=d(x).$$
\end{proof}

\begin{lemma}\label{MainLemma}
If $\gcd(x,y)=1$, then $D(xy) \leq D(x)D(y)$.
\end{lemma}

\begin{proof}
Consider the difference
$$D(x)D(y)-D(xy)=\bigg(2x-\sigma(x)\bigg)\bigg(2y-\sigma(y)\bigg)-\bigg(2xy-\sigma(xy)\bigg).$$
This is equal to
$$D(x)D(y)-D(xy)=4xy-2x\sigma(y)-2y\sigma(x)+\sigma(x)\sigma(y)-2xy+\sigma(x)\sigma(y)$$
since $\gcd(x,y)=1$.
Collecting like terms, we obtain
$$D(x)D(y)-D(xy)=2xy-2x\sigma(y)-2y\sigma(x)+2\sigma(x)\sigma(y)=2\cdot\bigg(xy-x\sigma(y)-y\sigma(x)+\sigma(x)\sigma(y)\bigg)$$
$$=2\cdot\bigg(\sigma(y)\cdot(\sigma(x)-x)-y\cdot(\sigma(x)-x)\bigg)=2\cdot\bigg(\sigma(x)-x\bigg)\cdot\bigg(\sigma(y)-y\bigg).$$
$D(xy) \leq D(x)D(y)$ now follows from $x \leq \sigma(x)$ and $y \leq \sigma(y)$ for all $x, y \in \mathbb{N}$.
\end{proof}

We are now ready to prove our claimed result.

\begin{theorem}\label{Theorem3}
Primes, prime powers, and products of two distinct odd prime powers are deficient.
\end{theorem}

\begin{proof}
We begin with the case of primes $q$.
$$d(q)=\frac{D(q)}{q}=\frac{2q-\sigma(q)}{q}=\frac{2q-(q+1)}{q}=\frac{q-1}{q}=1-\frac{1}{q}.$$
We compute
$$1+d(q)=2-\frac{1}{q}$$
$$2+d(q)=3-\frac{1}{q}.$$
Now we test whether the inequalities
$$\frac{2}{2-\frac{1}{q}}<I(q)=1+\frac{1}{q}<\frac{3-\frac{1}{q}}{2-\frac{1}{q}}$$
hold.
These inequalities are equivalent to
$$2<\bigg(1+\frac{1}{q}\bigg)\cdot\bigg(2-\frac{1}{q}\bigg)<3-\frac{1}{q}$$
which in turn are equivalent to
$$2<2+\frac{1}{q}-\bigg(\frac{1}{q}\bigg)^2<3-\frac{1}{q}$$
$$\bigg(0<\frac{1}{q}-\bigg(\frac{1}{q}\bigg)^2 = \frac{q-1}{q^2}\bigg)  \land \bigg(0<1-2\bigg(\frac{1}{q}\bigg)+\bigg(\frac{1}{q}\bigg)^2 = \bigg(\frac{q-1}{q}\bigg)^2\bigg).$$
Both inequalities are now readily seen to hold since $q$ prime implies that $q \geq 2 > 1$.  We therefore conclude, by Theorem \ref{Theorem1}, that primes are deficient.

We now consider the case of prime powers. Let $p$ be a prime and let $k$ be a positive integer.
$$d(p^k)=\frac{D(p^k)}{p^k}=\frac{2p^k-\sigma(p^k)}{p^k}=\frac{2p^k-(p^k+\sigma(p^{k-1}))}{p^k}=\frac{p^k-\sigma(p^{k-1})}{p^k}=1-\frac{\sigma(p^{k-1})}{p^k}.$$
Notice that the inequality
$$\sigma(p^{k-1})=\frac{p^k - 1}{p-1} < p^k$$
holds.
We compute
$$1+d(p^k)=2-\frac{\sigma(p^{k-1})}{p^k}$$
$$2+d(p^k)=3-\frac{\sigma(p^{k-1})}{p^k}.$$
Now we test whether the inequalities
$$\frac{2}{2-\frac{\sigma(p^{k-1})}{p^k}}<I(p^k)=\frac{\sigma(p^k)}{p^k}=\frac{p^k + \sigma(p^{k-1})}{p^k}=1+\frac{\sigma(p^{k-1})}{p^k}<\frac{3-\frac{\sigma(p^{k-1})}{p^k}}{2-\frac{\sigma(p^{k-1})}{p^k}}$$
hold.
These inequalities are equivalent to
$$2<2+\frac{\sigma(p^{k-1})}{p^k}-\bigg(\frac{\sigma(p^{k-1})}{p^k}\bigg)^2<3-\frac{\sigma(p^{k-1})}{p^k}$$
which in turn are equivalent to
$$\bigg(0<\frac{\sigma(p^{k-1})}{p^k}\bigg(1-\frac{\sigma(p^{k-1})}{p^k}\bigg)\bigg) \land \bigg(0<1-2\frac{\sigma(p^{k-1})}{p^k}+\bigg(\frac{\sigma(p^{k-1})}{p^k}\bigg)^2=\bigg(1 - \frac{\sigma(p^{k-1})}{p^k}\bigg)^2\bigg).$$
Both inequalities are now readily seen to hold since $\sigma(p^{k-1})<p^k$ implies that $1-\frac{\sigma(p^{k-1})}{p^k}>0$.  We therefore conclude, by Theorem \ref{Theorem1}, that prime powers are deficient.

Lastly, we turn our attention to products of two distinct odd prime powers.  Let $p$ and $q \neq p$ be primes, and let $r$ and $s$ be positive integers.
$$d({p^r}{q^s})=\frac{D({p^r}{q^s})}{{p^r}{q^s}}=\frac{2{p^r}{q^s}-\sigma({p^r}{q^s})}{{p^r}{q^s}}=2-I(p^r)I(q^s)$$
Notice that
$$1 < \bigg(1+\frac{1}{p}\bigg)\cdot\bigg(1+\frac{1}{q}\bigg) \leq I(p^r)I(q^s) < \bigg(1+\frac{1}{p-1}\bigg)\cdot\bigg(1+\frac{1}{q-1}\bigg) \leq \frac{3}{2}\cdot\frac{5}{4} = \frac{15}{8} < 2.$$
We compute
$$1+d({p^r}{q^s})=3-I(p^r)I(q^s)$$
$$2+d({p^r}{q^s})=4-I(p^r)I(q^s).$$
Now we test whether the inequalities
$$\frac{2}{3-I(p^r)I(q^s)}<I(p^r)I(q^s)<\frac{4-I(p^r)I(q^s)}{3-I(p^r)I(q^s)}$$
hold.
These inequalities are equivalent to
$$2<3I(p^r)I(q^s)-\bigg(I(p^r)I(q^s)\bigg)^2<4-I(p^r)I(q^s)$$
which in turn are equivalent to
$$\bigg(\bigg(I(p^r)I(q^s) - 1\bigg)\bigg(I(p^r)I(q^s) - 2\bigg) = \bigg(I(p^r)I(q^s)\bigg)^2 - 3I(p^r)I(q^s) + 2 < 0\bigg)$$
and
$$\bigg(\bigg(2-I(p^r)I(q^s)\bigg)^2= 4-4I(p^r)I(q^s) + \bigg(I(p^r)I(q^s)\bigg)^2 > 0\bigg),$$
which both imply that
$$I({p^r}{q^s})=I(p^r)I(q^s) < 2$$
since $I(p^r)I(q^s) > 1$.  Since $I(p^r)I(q^s)<2$ is known to be true, we therefore conclude by Theorem \ref{Theorem1} that products of two distinct odd prime powers are deficient.
\end{proof}

\begin{remark}\label{Remark1}
Why did we bother with a laborious proof for Theorem \ref{Theorem3}?  The method presented may lend itself well to further generalizations.
\end{remark}

\section{Friendly and Solitary Numbers in the Language of the Deficiency Index}
If there exists $y \neq x$ such that $I(x)=I(y)$, then
$$d(x)=2-I(x)=2-I(y)=d(y),$$
and $y$ is said to be a \emph{friend} of $x$.  (We shall likewise refer to $x$ and $y$ as \emph{friendly numbers}.)  Otherwise, if $I(x') \neq I(z)$ for all $z \in \mathbb{N}$, then
$$d(x')=2-I(x') \neq 2-I(z)=d(z),$$
for all $z \in \mathbb{N}$.  Such a number $x'$ is said to be \emph{solitary}.

We now show how to prove results for friendly and solitary numbers in the language of the deficiency index, similar to those that are done in terms of the abundancy index.

\begin{lemma}\label{Lemma1}
If $\gcd(n,D(n))=1$, then $n$ is solitary.
\end{lemma}

In particular, if the fraction $D(n)/n$ is in lowest terms, then $n$ is solitary by Lemma \ref{Lemma1}.

\begin{proof}
By Greening's Theorem \cite{Greening}, it suffices to show that
$$\gcd(n,D(n))=\gcd(n,\sigma(n)).$$
But
$$\gcd(n,D(n))=\gcd(n,2n-\sigma(n))=\gcd(n,\sigma(n)),$$
where we have used the fact that $\gcd(a,b)=\gcd(a,ax+by)$ for $x, y \in \mathbb{Z}$.
\end{proof}

\begin{corollary}\label{Corollary1}
Primes and powers of primes are solitary.
\end{corollary}

\begin{proof}
Let $q$ be a prime.  Then
$$D(q)=2q-\sigma(q)=2q-(q+1)=q-1,$$
which implies that $\gcd(q,D(q))=1$.  Hence, primes are solitary by Lemma \ref{Lemma1}.

Let $p$ be a prime, and let $k$ be a positive integer.  Then
$$D(q^k)=2q^k - \sigma(q^k)=2q^k - (q^k + \sigma(q^{k-1})) = q^k - \sigma(q^{k-1}).$$
We want to show that $\gcd(q^k,D(q^k))=1$.  Suppose to the contrary that
$$\gcd(q^k, D(q^k))=m>1.$$
Then $m \mid q^k$ and $m \mid D(q^k) = q^k - \sigma(q^{k-1})$.  It follows that $m \mid \sigma(q^{k-1})$, whence we have
$$\gcd(q^k, \sigma(q^{k-1})) \geq m > 1.$$
This is a contradiction.  We therefore conclude that $\gcd(q^k,D(q^k))=1$, so that prime powers are solitary.
\end{proof}

\begin{remark}\label{Remark2}
In particular, by Lemma \ref{Lemma1} and Corollary \ref{Corollary1}, there are infinitely many numbers $n$ satisfying $\gcd(n,D(n))=1$.
\end{remark}

\section{On Odd Deficient-Perfect Numbers}
A number $x$ is said to be \emph{deficient-perfect} if the divisibility condition $D(x) \mid x$ holds \cite{OEIS-A271816}.

$y = 9018009 = {3^2}\cdot{7^2}\cdot{{11}^2}\cdot{{13}^2}$ is deficient-perfect, since
$$D(y) = D(9018009) = 819 = {3^2}\cdot{7}\cdot{13}.$$

The quotient
$$\frac{y}{D(y)} = {7}\cdot{{11}^2}\cdot{13} = 11011$$
happens to be a palindrome!  By our formula relating the deficiency and abundancy indices, we have
$$\frac{D(y)}{y} = \frac{1}{11011}$$
and
$$I(y) = 2 - \frac{D(y)}{y} = \frac{22021}{11011}$$
which is perilously close to $2$ as some have described.

(\textbf{This portion is currently a work in progress.})

\section{Acknowledgments}
The author thanks the anonymous referee(s) whose valuable feedback improved the overall presentation and style of this manuscript.

\bibliographystyle{amsplain}

\end{document}